\documentclass[12pt]{article}
\usepackage{amssymb,amsthm,hyperref,amsmath,bm,amsfonts}
\usepackage{arydshln}
\usepackage{verbatim}
\usepackage{graphicx}
\usepackage{float}

\usepackage{subfigure}
\usepackage{multirow}
\usepackage{color}
\usepackage{appendix}

\usepackage{latexsym,amsmath,amsfonts,amscd}
\usepackage{epsfig}
\usepackage{changebar}
\usepackage{pstricks}
\usepackage{pst-plot}
\usepackage{multirow}
\usepackage{subfigure}
\usepackage{placeins}
\usepackage{amssymb}
\usepackage{enumerate}
\usepackage{tikz}
\usepackage{booktabs}

\newtheorem{theorem}{Theorem}[section]
\newtheorem{lemma}[theorem]{Lemma}

\newtheorem{definition}{Definition}[section]%

\newtheorem{remark}{Remark}[section]
\numberwithin{equation}{section}

\title{Optimal nodal control and the exponential turnpike property for the flow in gas networks}
\author{Martin Gugat \thanks{Chair for Dynamics, Control, Machine Learning and Numerics – Alexander von Humboldt Professorship, Department Mathematik, Friedrich-Alexander-Universität Erlangen-Nürnberg (FAU), Erlangen, Germany}, 
Michael Herty\thanks{Chair in Numerical Analysis, IGPM, RWTH Aachen University, Im Süsterfeld 2, D-52072 Aachen, Germany (herty@igpm.rwth-aachen.de); Extraordinary Professor, Department of Mathematics and Applied Mathematics, University of Pretoria, Private Bag X20, Hatfield 0028, South Africa} , Michael Schuster \thanks{Chair for Analytics \& Mixed-Integer Optimization, Department of Data Science, Friedrich-Alexander-Universität Erlangen-Nürnberg (FAU), Erlangen, Germany}, 
Yizhou Zhou \thanks{Corresponding author, IGPM, RWTH Aachen University, Im Süsterfeld 2, D-52072 Aachen, Germany (zhou@igpm.rwth-aachen.de)}}

\date{}
\begin{document}
\maketitle{}

\begin{abstract}
We study an optimal boundary control problem for a quasilinear hyperbolic system describing isothermal gas flow in pipelines. The control acts at network nodes and models the action of compressor stations. The system admits non-constant steady states, which are strictly monotone in space.

We work in the framework of $H^2$-solutions and consider an optimal control problem that steers the system toward a prescribed steady state. We  establish the existence of optimal controls and then prove the exponential turnpike property, showing that optimal trajectories remain exponentially close to the corresponding steady state over most of the time horizon.
The analysis relies on a cheap control property obtained via a Lyapunov approach and on a constructed control. Numerical simulations illustrate the theoretical result.
\end{abstract}

\section{Introduction}

The turnpike property for the optimal control of systems governed by partial differential equations has received considerable attention in recent years. 
The exponential turnpike property states that the optimal solution remains (exponentially) close to a reference solution for long time horizons. Usually, this reference solution is taken as the optimal solution to the corresponding static problem. 
For systems governed by ordinary differential equations, turnpike theory covers a broad class of general optimal control problems; see the surveys 
\cite{trelat2025turnpike},
and 
\cite{faulwasser2022turnpike, faulwasser2017turnpike}, which also discuss the time-discrete case. 

For systems governed by partial differential equations (PDEs),
results for distributed control are presented in \cite{porretta2013long}.
Applications often lead to boundary control problems in which the control operators are unbounded. In the case of linear--quadratic optimal control problems, the turnpike property for such settings was studied in 
\cite{nguyen2026turnpike, gugat2019turnpike, gugat2016optimal}.

In this paper, we consider the boundary control of a system governed by a quasilinear hyperbolic partial differential equation, which is not covered by the linear theory. More specifically, the system is described by the isothermal Euler equations modeling gas flow through pipelines (see \cite{gugat2022modeling}). In gas pipeline networks, the flow is driven by compressor stations. The corresponding control action is modeled as acting at a specific point in space, which we interpret as a node of the network graph. We refer to this setting as nodal control. The system under consideration is closely related to the one studied in \cite{MR2775185}.
 
Our main contribution is the proof of an exponential turnpike property for the optimal control problem. In our analysis, we work with $H^2$-solutions (see \cite{Coron-book}), which constitute a special class of semi-global classical solutions studied in \cite{MR2655971} and \cite{wang2006exact}.
The steady states of our system include functions that are strictly decreasing (or strictly increasing), and therefore their spatial derivatives do not vanish. Around such a steady state, we formulate an optimal control problem whose objective functional drives the flow of the time-dependent system toward the prescribed steady state. Assuming that the initial deviation is sufficiently small, we establish the existence of an optimal control.
To prove the exponential turnpike property, we first establish a cheap control property for the optimal control problem by means of a Lyapunov approach. As part of this analysis, we also introduce a constructed control for the system. This is motivated by the stabilizing nodal feedback law. Building on these results, we then derive the exponential turnpike property by following the approach developed in \cite{MR4955417}, which is validated and visualized by a numerical experiment for a network with two pipes and one compressor station. 

Feedback laws for hyperbolic systems with partially dissipative source terms have also been investigated in \cite{MR3500083,MR4759437,MR3895286} for linear systems and in \cite{wang2020boundary,MR4898367} for nonlinear problems. However, in these works the steady states are constant functions, whereas in our setting, due to the friction term in the PDE, the steady states are strictly decreasing (or strictly increasing) functions. In our work the feedback law serves solely as a tool for establishing the turnpike property of the associated optimal control problem.
The existence of optimal controls was studied in \cite{MR2516198} within the framework of BV weak solutions. In contrast, our existence result is established in the setting of $H^2$-solutions and our main focus is the proof of the exponential turnpike result. 


The rest of the paper is organized as follows. In Section \ref{section:2}, we formulate an optimal control problem for compressor-driven gas flow in pipelines. After introducing the model and its steady states, we prove the existence of an optimal control solution. Section \ref{section:3} is devoted to establishing the cheap control property by means of a Lyapunov method. Based on this, the exponential turnpike property is stated and proved in Section \ref{section:4}. The numerical experiments are employed in Section \ref{section:5} to validate the theoretical results. At last, we give concluding remarks in Section \ref{section:6}.

\section{The optimal control problem for the gas flow}\label{section:2}

In this section, we formulate an optimal control problem to model the operation of a compressor controlling gas flow in pipelines. We begin by introducing the underlying model, including the governing equations and the nodal conditions. Next, we analyze the steady state and describe how the time-dependent system evolves toward this equilibrium. Building on these results, we then formulate the optimal control problem and rigorously establish the existence of a solution.

\subsection{The model for pipeline gas flow}

We start with the
isothermal Euler equations that model the gas flow 
in a pipeline 
(see \cite{banda2006coupling}) 
\begin{align}
\rho_t+q_x&=0 \label{gas-eq1}\\[1mm]
q_t+\left(\frac{q^2}{\rho}+a^2\rho\right)_x &= -f_g\frac{q|q|}{2D\rho}. \label{gas-eq2}
\end{align}
Here $\rho=\rho(x,t)$ denotes the density of the gas, $q=q(x,t)$ is the mass flux in the pipe, $f_g$ is the friction factor and $D$ is
the diameter of the pipe.

We introduce the Riemann invariants
$$
R_+=-\frac{q}{\rho}-a\ln \rho,\qquad 
R_-=-\frac{q}{\rho}+a\ln \rho
$$
and rewrite the equation \eqref{gas-eq1}-\eqref{gas-eq2} as
\begin{align}\label{gas-RI-eq1}
\partial_t
\begin{pmatrix}
R_+\\[1mm]
R_-
\end{pmatrix}
+
\begin{pmatrix}
\lambda_+ &  0\\[1mm]
 0 & \lambda_-
\end{pmatrix}
\partial_x\begin{pmatrix}
R_+\\[1mm]
R_-
\end{pmatrix}=
\begin{pmatrix}
G\\[1mm]
G
\end{pmatrix}
\end{align}
with 
$$
\lambda_\pm =\lambda_\pm(R_+,R_-)= -\frac{R_++R_-}{2}\pm a,\qquad 
G=G(R_+,R_-) = \frac{f_g}{8D}(R_++R_-)^2.
$$
Throughout this paper, we
only consider the subsonic flow. Namely,
\begin{equation}\label{subsonic}
0<\frac{
q}{\rho}<a.
\end{equation}
This gives $\lambda_+ \, \lambda_- < 0$.
Then at each pipe there is one incoming
Riemann invariant and one outgoing Riemann invariant. Note that under normal operating conditions of gas-transportation systems (see, e.g., \cite{KochEtAl2015, GugatEtAl2015, Gas-Modell:DomschkeHillerLangTischendorf2017, UlkeEtAl2025}), this assumption is satisfied.


\subsection{Node and boundary conditions}

In order to treat networks of pipes controlled by compressors, we consider the system as that in \cite{MR2775185}. Details are depicted in  Figure \ref{fig1}. Namely, two pipes are connected by a node in the middle. Denote $(\rho^{(i)},q^{(i)})$ to be solutions in the pipe $i=1,2$ respectively. Then $(\rho^{(i)},q^{(i)})$ satisfies the equations \eqref{gas-eq1}-\eqref{gas-eq2} on $(0,L_i)$.
In the middle node, we consider a compressor $u_0=u_0(t)$ such that:
\begin{align}
    q^{(1)}(L_1,t)&=q^{(2)}(0,t) \label{couplingMass} \\[1mm]
    u_0(t)&=q^{(2)}(0,t)\left[\left(\frac{\rho^{(2)}(0,t)}{\rho^{(1)}(L_1,t)} \right)^\kappa -1\right]. \label{couplingCompressor}
\end{align}
Here $\kappa\in[\frac{1}{3}, \frac{3}{5}]$ is the isentropic exponent, which depends on the gas under consideration. The compressor power $u_0=u_0(t)$ serves as the control variable. Moreover, mass conservation requires the flow to be constant along the compressor edges. In this work, we assume that the pipe satisfies the inequality $q^{(1)}(L_1,t)=q^{(2)}(0,t) \geq 0$ , i.e., the gas flows from
the first pipe into the second pipe.

\begin{figure}[h!]
\begin{center}
\begin{tikzpicture}[
    every node/.style={circle, draw, minimum size=1mm}, >=stealth,
    thick
]

\node (u1) at (0,0) {$u_1$};
\node (u2) at (7,0) {$u_2$};
\node (u0) at (3.2,0) {$u_0$};

\node[draw=none, above] at (1.6,-0.5) {pipe $(1)$};
\node[draw=none, above] at (5.1,-0.5) {pipe $(2)$};
\node[draw=none, below] at (3.2,0.2) {compressor};

\draw[->] (u1) -- (u0);
\draw[->] (u0) -- (u2);

\end{tikzpicture}
\end{center}
\caption{Structure of the control system for a compressor station with two pipes.}
\label{fig1}
\end{figure}
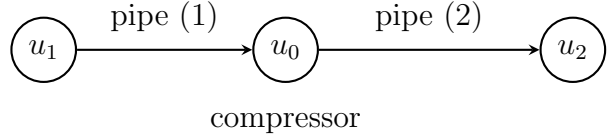

Using the Riemann invariants, we can also write the node conditions as 
\begin{align}
    R_-^{(1)}(L_1,t)&=\Phi_1(R_+^{(1)}(L_1,t),R_-^{(2)}(0,t),u_0(t)) \label{gas-RI-eq2}\\[1mm]
    R_+^{(2)}(0,t)&=\Phi_2(R_+^{(1)}(L_1,t),R_-^{(2)}(0,t),u_0(t)). \label{gas-RI-eq3}
\end{align}
Here $\Phi_1$ and $\Phi_2$ are two smooth functions of $(R_+^{(1)}(L_1,t),R_-^{(2)}(0,t),u_0(t))$. Moreover, for the first pipe at $x=0$ and for the second pipe at $x=L_2$, we consider the boundary conditions 
\begin{align}
    R_+^{(1)}(0,t)&=u_1(t)\label{gas-RI-eq4}\\[1mm]
    R_-^{(2)}(L_2,t)&=u_2(t) \label{gas-RI-eq5}.
\end{align}
Moreover, the initial data for $R^{(i)}=(R^{(i)}_+,R^{(i)}_-)$ are given by
\begin{equation}\label{R-initial-data}
    R^{(i)}(x,0)=R^{(i)}_0(x)=(R^{(i)}_{0,+},R^{(i)}_{0,-})(x),\qquad R_{0,\pm}^{(i)}(x) = -\frac{q^{(i)}_0}{\rho^{(i)}_0}\mp a\ln \rho^{(i)}_0,
\end{equation}
where $\rho_0^{(i)}$ and $q_0^{(i)}$ are initial data for the density and flux in the $i$-th pipe:
$$
\rho^{(i)}(x,0)=\rho^{(i)}_0(x),\qquad q^{(i)}(x,0)=q^{(i)}_0(x).
$$
In what follows, for simplicity, we assume that $L_1=L_2=L$ without loss of generality. 

\subsection{Steady states}

In this subsection, we discuss the static controls $\bar{u}_i~(i=0,1,2)$ such that equations \eqref{gas-RI-eq1}-\eqref{gas-RI-eq5} admit an intransient  solution $\bar{R}^{(i)}=(\bar{R}_+^{(i)},\bar{R}_-^{(i)})$. 
In this work, we follow the framework in \cite{MR2775185} and state 
\begin{lemma}
Let $\bar{\rho}_0>0$ and $\bar{q}_0>0$ be two positive constants such that $\bar{q}_0/\bar{\rho}_0<a$. Then there exists a unique $C^2$-stationary subsonic solution $(\bar{\rho}(x), \bar{q}(x))$ on the interval $[0, \bar{L})$
that verifies the boundary conditions
$(\bar{\rho}(0), \bar{q}(0))=(\bar{\rho}_0, \bar{q}_0)$. Here $\bar{L}$ is the maximum length for the existence with 
$$
\bar{L}=\bar{L}(\bar{\rho}_0,\bar{q}_0) = \frac{D}{f_g\bar{q}_0^2}(F(\bar{\rho}_0)-F(\bar{q}_0/a))
$$
and $F(\bar{\rho})=-2 \bar{q}^2\ln(\bar{\rho}) +a^2\bar{\rho}^2$.
\end{lemma}
\begin{proof}
We start with the original form \eqref{gas-eq1}-\eqref{gas-eq2}. Clearly, for the steady state $(\rho,q)\equiv (\bar{\rho},\bar{q})$, the mass flux $\bar{q}=\bar{q}_0$ is just a constant and $\bar{\rho}$ is a solution to the ODE 
$$
\left(\frac{\bar{q}^2}{\bar{\rho}}+a^2\bar{\rho}\right)_x = -f_g\frac{\bar{q}|\bar{q}|}{2D\bar{\rho}},\qquad \bar{\rho}(0)=\bar{\rho}_0.
$$
For $\bar{\rho}>0$, we derive from the ODE system that 
$$
\left(-2 \bar{q}^2\ln(\bar{\rho}) +a^2\bar{\rho}^2\right)_x = -f_g\frac{\bar{q}^2}{D}~~\Rightarrow~~ F(\bar{\rho}) = F(\bar{\rho}(0)) -f_g\frac{\bar{q}^2}{D}x.
$$
For $\bar{\rho}\in(\bar{q}_0/a,\infty)$, we compute
$$
F'(\bar{\rho})=-\frac{2\bar{q}^2}{\bar{\rho}}+2a^2\bar{\rho}>0.
$$
Thus there exists a strictly increasing function $H(\cdot)=F^{-1}(\cdot)$ such that 
$$
\bar{\rho}(x)=H\left(F(\bar{\rho}_0) -f_g\frac{\bar{q}^2}{D}x\right).
$$
Since $\bar{q}>0$, we know that $\bar{\rho}(x)$ is strictly decreasing w.r.t. $x>0$. Thus the density $\bar{\rho}(x)>\bar{q}_0/a$ if and only if 
$$
F(\bar{\rho}(x))= F(\bar{\rho}_0) -f_g\frac{\bar{q}^2}{D}x > F(\bar{q}_0/a).
$$
This gives the critical length and completes the proof.
\end{proof}

Next, we consider the node problem where our system consists of two pipes coupled through a compressor.
We give $\bar{\rho}^{(1)}(0)=\bar{\rho}_0$ as a constant at the left node. Due to the coupling condition $q^{(1)}(L,t)=q^{(2)}(0,t)$, we
have the same constant flux $\bar{q}=\bar{q}_0$ on both pipes. On the first pipe, we prescribe the constant control 
$$
\bar{u}_1 = -\frac{\bar{q}_0}{\bar{\rho}_0}-a\ln \bar{\rho}_0
$$
such that the boundary condition \eqref{gas-RI-eq4} is fulfilled. For sufficiently large $\bar{\rho}_0$, if $L<\bar{L}(\bar{\rho}_0,\bar{q}_0)$, then there exists a $C^2$ solution.  According to the  coupling condition of the compressor, we have $$\bar{\rho}^{(2)}(0) = \bar{\rho}^{(1)}(L) (\bar{u}_0/\bar{q}_0 + 1)^{1/\kappa}$$ where $\bar{u}_0 \geq 0$ is the constant compressor
control. For sufficiently large $\bar{u}_0$, we see that $\bar{\rho}^{(2)}(0)$ is also large enough such that $L<\bar{L}(\bar{\rho}^{(2)}(0),\bar{q}_0)$. In this case, we also have the $C^2$ solution $\bar{\rho}^{(2)}(x)$ in the second pipe. At last, we set the constant control 
$$
\bar{u}_2 = -\frac{\bar{q}_0}{\bar{\rho}^{(2)}(L)}+a\ln (\bar{\rho}^{(2)}(L))
$$
to fit the boundary condition \eqref{gas-RI-eq5}.

\subsection{Nonstationary system}
After analyzing the steady-state formulation, we now turn to the time-dependent problem, which describes the evolution of the system toward steady state.
In order to consider the deviation between the solution of the dynamic problem and the steady state, we denote 
$$
r^{(i)}  = R^{(i)}  - \bar{R}^{(i)} 
$$
with 
$$
r^{(i)}(x,t)=\left(r_{+}^{(i)},r_{-}^{(i)}\right)(x,t),\quad 
R^{(i)}(x,t)=\left(R_{+}^{(i)},R_{-}^{(i)}\right)(x,t),\quad 
\bar{R}^{(i)}(x)=\left(\bar{R}_{+}^{(i)},\bar{R}_{-}^{(i)}\right)(x).
$$
Then we derive from \eqref{gas-RI-eq1} that
\begin{align}
\partial_tr_+^{(i)} 
+\left(\bar{\lambda}^{(i)}_+ - \dfrac{r_+^{(i)}+r_-^{(i)}}{2}\right)\partial_xr_+^{(i)}
=-\left(r_+^{(i)}+r_-^{(i)}\right)
\left(K_+-\frac{f_g}{8D}\left(r_+^{(i)}+r_-^{(i)}\right)\right)\label{gas-r-eq1}\\[2mm]
\partial_tr_-^{(i)} 
+\left(\bar{\lambda}^{(i)}_- - \dfrac{r_+^{(i)}+r_-^{(i)}}{2}\right)\partial_xr_-^{(i)}
=-\left(r_+^{(i)}+r_-^{(i)}\right)
\left(K_--\frac{f_g}{8D}\left(r_+^{(i)}+r_-^{(i)}\right)\right)\label{gas-r-eq2}
\end{align}
with 
$$
\bar{\lambda}^{(i)}_\pm = -\frac{\bar{R}^{(i)}_++\bar{R}^{(i)}_-}{2}\pm a=\frac{\bar{q}^{(i)}}{\bar{\rho}^{(i)}}\pm a,
\quad
K_\pm = -\frac{f_g}{8D}\left(
\bar{R}_+^{(i)}+\bar{R}_-^{(i)}\right)
\frac{4a\mp
(\bar{R}_+^{(i)}+\bar{R}_-^{(i)})}{2a\mp(
\bar{R}_+^{(i)}+\bar{R}_-^{(i)})}.
$$
Due to the positive mass flux assumption and the subsonic condition \eqref{subsonic}, we have 
$$
-a<\frac{\bar{R}_+^{(i)}+\bar{R}_-^{(i)}}{2}<0
$$
and thereby obtain 
$K_\pm(x)>0$
for any $x\in [0,L]$.

Now we turn to discuss the boundary conditions and coupling condition. From \eqref{gas-RI-eq2}--\eqref{gas-RI-eq5}, we derive
\begin{align}
    r_-^{(1)}(L,t)&=\Phi_1\left({r}_+^{(1)}(L,t)+\bar{R}_+^{(1)}(L),{r}_-^{(2)}(0,t)+\bar{R}_-^{(2)}(0),u_0(t)\right)\notag \\[1mm]
    &\quad -\Phi_1(\bar{R}_+^{(1)}(L),\bar{R}_-^{(2)}(0),\bar{u}_0) \label{eq1.11}\\[2mm]
    r_+^{(2)}(0,t)&=\Phi_2\left({r}_+^{(1)}(L,t)+\bar{R}_+^{(1)}(L),{r}_-^{(2)}(0,t)+\bar{R}_-^{(2)}(0),u_0(t)\right)\notag \\[1mm]
    &\quad -\Phi_2(\bar{R}_+^{(1)}(L),\bar{R}_-^{(2)}(0),\bar{u}_0)  \label{eq1.12}
\end{align}
and
\begin{align}\label{eq1.13}
    r_+^{(1)}(0,t)=u_1(t)-\bar{u}_1,\qquad
    r_-^{(2)}(L,t)=u_2(t)-\bar{u}_2.
\end{align}
According to \eqref{R-initial-data}, the initial data are given by 
\begin{equation}\label{deviation-initial-data}
    r^{(i)}(x,0)=r_{0}^{(i)}(x)=R_{0}^{(i)}(x)-\bar{R}^{(i)}(x).
\end{equation}
We assume that the initial data $r_{0}^{(i)}(x)$ are given in the Sobolev space $H^2(0,L)$ and satisfy
\begin{equation}\label{initial data}
\|r_{0}^{(i)}\|_{H^2(0,L)}\leq \epsilon_{in}(T),\qquad i=1,2
\end{equation}
with $\epsilon_{in}(T)>0$ a small constant. The constant $\epsilon_{in}(T)$ is chosen sufficiently small for two
purposes. First, it guarantees the existence of an $H^2$-solution on
the prescribed interval $[0,T]$. Second, it ensures that every
intermediate state of an optimal solution remains in the 
neighborhood of the stationary state in which the Lyapunov and
constructed control arguments of Section~3 apply (see Remark \ref{remark32}).


In addition, we assume that the initial data $r_0=(r_0^{(1)},r_0^{(2)})$ satisfy the zeroth- and first-order compatibility conditions associated with
the boundary and nodal conditions
\eqref{eq1.11}--\eqref{eq1.13}. More precisely, there exist
initial control values
\[
u_i(0) = \bar{u}_i+u_i^{\mathrm{in}},\qquad \partial_t u_i(0)=\dot u_i^{\mathrm{in}},
\qquad i=0,1,2,
\]
such that the boundary and nodal conditions and their first time
derivatives are satisfied at $t=0$. 

\subsection{The optimal control problem}
Let a time-horizon $T>0$ be given. 
Consider the optimal control problem with the objective function
\begin{align}\label{objective-func}
J_T(u_0, u_1, u_2)= \sum_{i=1}^2\int_0^T \|r^{(i)}(\cdot,t)\|^2_{H^2(0,L)} dt + \sum_{i=0}^2 \|u_i-\bar{u}_i\|_{H^2(0,T)}^2.
\end{align}
Here $r^{(i)}=r^{(i)}(x,t)$ with $i=1,2$ satisfies equations and boundary conditions \eqref{gas-r-eq1}--\eqref{eq1.13}, $u_{i}=u_{i}(t)$ with $i=0,1,2$ are control terms, and the notation 
$H^2(0,T)$ represents the Sobolev norm as that in \cite{Coron-book}.
The objective function enforces the desired controls $\bar{u}_0,\bar{u}_1,\bar{u}_2$ and the desired steady state $\bar{R}^{(i)}$ for pipe $(1)$ and pipe $(2)$.

In the  change to renewable energies
also hydrogen is transported in the gas pipelines,
both as pure hydrogen and in a mixture with natural gas.
In both cases it is possible that the
steel of the pipelines is damaged
by hydrogen embrittlement.
To mitigate this effect, 
pressure fluctuations should be avoided, see \cite{laureys2022use,xing2017atomistic}.
The $H^2$ norm in the  objective function enforces  flows 
that are suitable for such a situation. 

In order to define the optimal control problem, we introduce the set $\mathcal{U}_T$ of feasible controls. 
\begin{definition}
The feasible set $\mathcal{U}_T$ contains the control functions in $H^2(0,T)$ such that the boundary conditions are $C^1$-compatible with the initial state \eqref{deviation-initial-data} and the following control constraint is satisfied: 
\begin{align}\label{control-constraint1}
\|u_i -\bar{u}_i\|_{H^2(0,T)}\leq \epsilon(T),\qquad i=0,1,2.
\end{align}
\end{definition}

We take $\epsilon(T)$ sufficiently small. Moreover here we also take the initial  parameter $\epsilon_{in}(T)\leq \epsilon(T)$ such that the following properties hold: 
\begin{enumerate}
    \item[(1)] With constraints \eqref{initial data}--\eqref{control-constraint1} for the initial and boundary data, there exists a unique solution for the equation \eqref{gas-r-eq1}--\eqref{gas-r-eq2} in the space
$$
CH_T^2=\cap_{k=0}^2 C^k([0,T];H^{2-k}(0,L)).
$$ 
Moreover, the boundary data belong to the space $H^2(0,T)$.
This is true for sufficiently small $\epsilon(T)>0$ by the existence theory in Sobolev space \cite{Coron-book,ExistenceHyperbolic}.
\item[(2)] By the embedding theory, the $C^1$-norm of the initial and boundary data are also sufficiently small. Namely, 
\begin{align}\label{control-constraint2}
\max_{x\in [0,L]}\{ |r^{(i)}(x,t)|,|\partial_x r^{(i)}(x,t)|\}\leq C_T~\epsilon(T),\qquad i=1,2
\end{align}
and
\begin{align}\label{control-constraint3}
\max_{t\in[0,T]}\{ |u_i(t)-\bar{u}_i|,|\partial_t u_i(t)|\}\leq C_T~\epsilon(T),\qquad i=0,1,2.
\end{align}
Here $C_T>0$ is a positive constant.
Then by the theory of semi-global classical solutions \cite{MR2655971},  there exists a classical solution $r^{(i)}(x,t)\in C^1([0,T]\times [0,L])$ for $i=1,2$.
\end{enumerate}


Having these, we state the following existence result: 
\begin{theorem}
Assume that the set $\mathcal{U}_T$ of admissible controls is non-empty. Let $T > 0$ be given.
Consider the dynamic optimal control problem $\mathcal{P}(T)$:
\begin{align*}
\min_{(u_0, u_1, u_2)\in \mathcal{U}_T} &
J_T(u_0, u_1, u_2)
\end{align*}
with $J_T$ as defined in \eqref{objective-func}.
Here $r^{(i)}$ satisfies the partial differential equations \eqref{gas-r-eq1}--\eqref{gas-r-eq2} with initial condition \eqref{deviation-initial-data} and boundary conditions \eqref{eq1.11}--\eqref{eq1.13}. Then there exists an optimal solution to the problem $\mathcal{P}(T)$. 
\end{theorem}


\begin{proof}
In the proof, we represent $r=(r^{(1)},r^{(2)})$ and $u=(u_0,u_1,u_2)$ for simplicity. 

We consider a minimizing sequence of admissible controls $\{u_{n}\}_{n\geq 1}$ for $\mathcal{P}(T)$ and denote $r_n$ the corresponding states. Due to the cost function, the sequence of controls $\{u_{n}\}_{n\geq 1}$ is bounded with respect to the norm $\|\cdot\|_{H^2(0,T)}$. 
Therefore, we have a subsequence such that $u_{n_k}$ weakly converges in $H^2(0,T)$ and strongly converges in $C^1(0,T)$, to be precise, 
$$
u_{n_k} \rightharpoonup u^{*} ~~\text{weakly in}~H^2(0,T),\qquad 
u_{n_k} \rightarrow u^{*} ~~\text{strongly in}~C^1(0,T).
$$
Based on the continuous-dependence result for semi-global solutions (see e.g. \cite{wang2006exact}), we know that the subsequence $r_{n_k}$ satisfies 
\begin{align*}
&r_{n_k} \rightarrow r^{*} ~~\text{strongly in}~C^1([0,T]\times [0,L])
\end{align*}
where $r^{*}$ is the solution generated by the limit point $u^*$. The strong convergence of $u_{n_k}$ in $C^1$ also implies that $u^*$ is $C^1$-compatible with the initial data. Additionally, by the weakly lower-semicontinuity property of the norm for the Hilbert space $H^2(0,T)$, we have
\begin{equation}\label{proof-ex-eq1}
\|u^*-\bar{u}\|_{H^2(0,T)} \leq  \liminf_{k\rightarrow \infty} \|u_{n_k}-\bar{u}\|_{H^2(0,T)}\leq \epsilon(T).
\end{equation}
Thus $u^*$ satisfies the constraints \eqref{control-constraint1}---\eqref{control-constraint2} and thereby is in the admissible set $\mathcal{U}_T$. 

Since $r_{n_k} \rightarrow r^{*}$ in $C^1([0,T]\times [0,L])$, we have $r_{n_k} \rightarrow r^{*}$ in $L^2([0,T]\times [0,L])$. On the other hand, since the term $$\int_0^T\| r_{n_k}\|^2_{H^2(0,L)}dt$$ in the cost functional is uniformly bounded, we know that $r_{n_k} \rightharpoonup r^{**}$ in $L^2(0,T; H^2(0,L))$ after extracting a further subsequence. The space $L^2(0,T; H^2(0,L))$ is continuously embedded in $L^2([0,T]\times [0,L])$ and therefore we have $r_{n_k} \rightharpoonup r^{**}$ in $L^2([0,T]\times [0,L])$. Consequently, $r^*=r^{**}$.
Thanks to this, we have 
$$
\int_0^T \|r^* \|^2_{H^2(0,L)} dt  \leq  \liminf_{k\rightarrow \infty} \int_0^T\| r_{n_k}\|^2_{H^2(0,L)}dt.
$$
Combining this with \eqref{proof-ex-eq1}, we have
\begin{align*}
J_T(u^*)=&~\int_0^T \|r^* \|^2_{H^2(0,L)} dt+\|u^*-\bar{u}\|_{H^2(0,T)}^2 \\[2mm]
\leq &~ \liminf_{k\rightarrow \infty} \Big[\int_0^T \|r_{n_k} \|^2_{H^2(0,L)} dt + \|u_{n_k}-\bar{u}\|_{H^2(0,T)}^2\Big]=\liminf_{k\rightarrow \infty}J_T(u_{n_k}).
\end{align*}
Since $(u_n,r_n)$ is a minimizing sequence, we conclude that $(u^*,r^*)$ is an optimal solution.

\end{proof}

\section{The cheap control property}\label{section:3}
Similar to that in \cite{gugat2024turnpike}, 
our turnpike analysis also uses the following 
cheap control property:
\begin{definition}[Cheap control]
    The optimal control problem $\mathcal{P}(T)$ satisfies the cheap control property if there exists a constant $C_0>0$ (independent of $T$) such that, for all $T>0$, each optimal control for 
    $\mathcal{P}(T)$ satisfies the inequality 
    \begin{align}\label{cheap-control-prop}
        J_T(u)=&\sum_{i=1}^2\int_0^T \|r^{(i)}(\cdot,t) \|^2_{H^2(0,L)} dt + \sum_{i=0}^2 \|u_i-\bar{u}_i\|_{H^2(0,T)}^2
        \leq C_0 \sum_{i=1}^2 \|r^{(i)}_0\|^2_{H^2(0,L)}. 
    \end{align}
    Here $u(t)=(u_0(t),u_1(t),u_2(t))$ is the optimal control, $r^{(i)}(x,t)$ represents the corresponding state, and $r^{(i)}_0$ is the initial data in \eqref{deviation-initial-data}.
\end{definition}
The following subsections are devoted to the proof of the cheap control property for the optimal control problem $\mathcal{P}(T)$. We begin with an overview of the strategy and a sketch of the main steps to guide the reader through the argument.
\begin{itemize}
    \item[(1)] The first step is to obtain an estimate for the objective function with arbitrary feasible control. This can be done by exploiting a Lyapunov method.
    \item [(2)] Particularly, we take a constructed control as a feasible control and substitute it into the estimate in the first step. The objective function with this constructed control can be bounded by the initial data uniformly with respect to $T$. 
    \item [(3)] Due to the optimality, the objective function with the optimal control can be also bounded by the initial data. This implies the cheap control property.
\end{itemize}
\subsection{$H^2$-Lyapunov functional}
Similarly as in \cite{MR2996527}, 
define the Lyapunov functional $\mathcal{E}(t)=\mathcal{E}_0(t)+\mathcal{E}_1(t)+\mathcal{E}_2(t)$ by 
\begin{align*}
    \mathcal{E}_0(t) = \sum_{j=1}^2\mathcal{E}_0^{(j)}(t)= \sum_{j=1}^2 \int_0^L\frac{A^{(i)}}{\bar{\lambda}^{(i)}_+(x)}h^{(i)}_+(x) |r_+^{(i)}|^2 + \frac{B^{(i)}}{|\bar{\lambda}^{(i)}_-(x)|}h^{(i)}_-(x) |r_-^{(i)}|^2~dx
\end{align*}
where 
$$
h^{(i)}_\pm(x) = \exp\left(-\mu \int_0^x\frac{1}{\bar{\lambda}_\pm(s)}ds\right)\qquad
\frac{1}{\mu} = \int_0^L\left(\frac{1}{\bar{\lambda}_+(x)}+\frac{1}{|\bar{\lambda}_-(x)|} \right)dx.
$$
Moreover, we introduce
\begin{align*}
    \mathcal{E}_1(t) = \sum_{j=1}^2\mathcal{E}_1^{(j)}(t)= \sum_{j=1}^2 \int_0^L C^{(i)}h^{(i)}_+(x) |p_+^{(i)}|^2 +  D^{(i)} h^{(i)}_-(x) |p_-^{(i)}|^2~dx\\[1mm]
    \mathcal{E}_2(t) = \sum_{j=1}^2\mathcal{E}_2^{(j)}(t)= \sum_{j=1}^2 \int_0^L C^{(i)}h^{(i)}_+(x) |s_+^{(i)}|^2 +  D^{(i)} h^{(i)}_-(x) |s_-^{(i)}|^2~dx
\end{align*}
with $$p_{\pm}(x,t)=\partial_tr_{\pm}(x,t),\qquad s_{\pm}(x,t)=\partial_tp_{\pm}(x,t).$$ 

The following result follows from \cite{MR2996527}:
\begin{lemma}[\cite{MR2996527}]
For sufficiently small $L$, there exists a sufficiently large $c_L>0$. We can choose the positive weights
$A^{(i)}$, $B^{(i)}$, $C^{(i)}$ and $D^{(i)}$ with $i=1,2$ satisfying
\begin{itemize}
    \item[(i)] $\dfrac{A^{(i)}}{B^{(i)}}\geq c_L$ or $\dfrac{B^{(i)}}{A^{(i)}}\geq c_L$
    \item[(ii)] $\dfrac{C^{(i)}}{D^{(i)}}\geq c_L$ or $\dfrac{D^{(i)}}{C^{(i)}}\geq c_L$
\end{itemize}
such that the Lypunov functional $\mathcal{E}(t)=\mathcal{E}_0(t)+\mathcal{E}_1(t)+\mathcal{E}_2(t)$ satisfies 
\begin{align}\label{Ly-estimate}
    \frac{d}{dt}\mathcal{E}(t) \leq -\mu \mathcal{E}(t)+\sum_{j=1}^2\left(I^{(j)}_0+I^{(j)}_L\right)
\end{align}
with the boundary terms given by
\begin{align}
I^{(i)}_0 &= A^{(i)} |r_+^{(i)}(0,t)|^2-B^{(i)} |r_-^{(i)}(0,t)|^2\nonumber\\[1mm]
&~ +C^{(i)}\Big(\bar{\lambda}^{(i)}_+(0)-\frac{r^{(i)}_+(0,t)+r^{(i)}_-(0,t)}{2}\Big)\left(|p_+^{(i)}(0,t)|^2+|s_+^{(i)}(0,t)|^2\right)\nonumber\\
&~ -D^{(i)}\Big(|\bar{\lambda}^{(i)}_-(0)|+\frac{r^{(i)}_+(0,t)+r^{(i)}_-(0,t)}{2}\Big)\left(|p_-^{(i)}(0,t)|^2+|s_-^{(i)}(0,t)|^2\right)\nonumber
\end{align}
and
\begin{align}
I^{(i)}_L &= -A^{(i)}h_+^{(i)}(L)|r_+^{(i)}(L,t)|^2+B^{(i)}h_-^{(i)}(L)|r_-^{(i)}(L,t)|^2\nonumber\\[1mm]
&~-C^{(i)}h_+^{(i)}(L)\Big(\bar{\lambda}^{(i)}_+(L)+\frac{r^{(i)}_+(L,t)+r^{(i)}_-(L,t)}{2}\Big)\left(|p_+^{(i)}(L,t)|^2+|s_+^{(i)}(L,t)|^2\right)\nonumber\\
&~+ D^{(i)}h_-^{(i)}(L)\Big(|\bar{\lambda}^{(i)}_-(L)|-\frac{r^{(i)}_+(L,t)+r^{(i)}_-(L,t)}{2}\Big)\left(|p_-^{(i)}(L,t)|^2+|s_-^{(i)}(L,t)|^2\right).\nonumber
\end{align}
\end{lemma}

Note that the Lyapunov method and the notations are similar to those in \cite{MR2996527}. 
For the boundary terms, we prove
\begin{lemma}
For sufficiently small $L$, there exist constants $A^{(i)}$, $B^{(i)}$, $C^{(i)}$ and $D^{(i)}$ such that the following relation holds 
\begin{align}\label{lemma3.2-eq1} 
\sum_{j=1}^2\left(I^{(j)}_0+I^{(j)}_L\right)\leq C\sum_{i=0}^2\Big(|u_i(t)-\bar{u}_i|^2+|\partial_tu_i(t)|^2+|\partial_t^2u_i(t)|^2\Big).
\end{align}
Here $C>0$ is a generic constant.
\end{lemma}
\begin{proof}
Recall that the boundary conditions of $r^{(i)}(x,t)$ are given by \eqref{eq1.11}--\eqref{eq1.13}. We take time derivatives on these relations to obtain
\begin{align}
    p_-^{(1)}(L,t)&=\partial_{R_+}\Phi_1\cdot p_+^{(1)}(L,t) + \partial_{R_-}\Phi_1\cdot p_-^{(2)}(0,t)  +\partial_{u}\Phi_1\cdot \partial_tu_0(t)\label{eq-p1}\\[2mm]
    p_+^{(2)}(0,t)&=\partial_{R_+}\Phi_2\cdot p_+^{(1)}(L,t) + \partial_{R_-}\Phi_2\cdot p_-^{(2)}(0,t)  +\partial_{u}\Phi_2\cdot \partial_tu_0(t)\label{eq-p2}
\end{align}
and
\begin{align}
    p_+^{(1)}(0,t)=\partial_tu_1(t),\qquad
    p_-^{(2)}(L,t)=\partial_tu_2(t).
\end{align}
Moreover, differentiating the coupling conditions twice with
respect to time, we obtain
\begin{align}
s_-^{(1)}(L,t)
={}&
\partial_{R_+}\Phi_1\,s_+^{(1)}(L,t)
+\partial_{R_-}\Phi_1\,s_-^{(2)}(0,t)
+\partial_u\Phi_1\,\partial_t^2u_0(t)
\nonumber\\
&+
\partial_{R_+R_+}\Phi_1
\left|p_+^{(1)}(L,t)\right|^2
+
\partial_{R_-R_-}\Phi_1
\left|p_-^{(2)}(0,t)\right|^2
+
\partial_{uu}\Phi_1
\left|\partial_tu_0(t)\right|^2
\nonumber\\
&+
2\partial_{R_+R_-}\Phi_1
p_+^{(1)}(L,t)p_-^{(2)}(0,t)
\nonumber\\
&+
2\partial_{R_+u}\Phi_1
p_+^{(1)}(L,t)\partial_tu_0(t)
+
2\partial_{R_-u}\Phi_1
p_-^{(2)}(0,t)\partial_tu_0(t),
\label{eq-s1}
\end{align}
and
\begin{align}
s_+^{(2)}(0,t)
={}&
\partial_{R_+}\Phi_2\,s_+^{(1)}(L,t)
+\partial_{R_-}\Phi_2\,s_-^{(2)}(0,t)
+\partial_u\Phi_2\,\partial_t^2u_0(t)
\nonumber\\
&+
\partial_{R_+R_+}\Phi_2
\left|p_+^{(1)}(L,t)\right|^2
+
\partial_{R_-R_-}\Phi_2
\left|p_-^{(2)}(0,t)\right|^2
+
\partial_{uu}\Phi_2
\left|\partial_tu_0(t)\right|^2
\nonumber\\
&+
2\partial_{R_+R_-}\Phi_2
p_+^{(1)}(L,t)p_-^{(2)}(0,t)
\nonumber\\
&+
2\partial_{R_+u}\Phi_2
p_+^{(1)}(L,t)\partial_tu_0(t)
+
2\partial_{R_-u}\Phi_2
p_-^{(2)}(0,t)\partial_tu_0(t).
\label{eq-s2}
\end{align}
and
\begin{align}
    s_+^{(1)}(0,t)=\partial_t^2u_1(t),\qquad
    s_-^{(2)}(L,t)=\partial_t^2u_2(t).
\end{align}
Since $\Phi_1$ and $\Phi_2$ are of class $C^2$ and their
arguments remain in a compact set, all first- and second-order
derivatives of $\Phi_1$ and $\Phi_2$ are uniformly bounded.
Moreover, the smallness assumptions on the solution and controls,
together with the equations, imply that
\[
|p_+^{(1)}(L,t)|
+|p_-^{(2)}(0,t)|
+|\partial_tu_0(t)|
\leq \delta
\]
for some small constant $\delta>0$. Thus,
\[
|p_+^{(1)}(L,t)|^4
\leq
\delta^2|p_+^{(1)}(L,t)|^2,
\]
and analogous estimates hold for
$p_-^{(2)}(0,t)$ and $\partial_tu_0(t)$. The mixed terms are
controlled by $2|ab|\leq |a|^2+|b|^2$. Hence, there exists a constant $\widehat C_0>0$  such that
\begin{align}
|s_-^{(1)}(L,t)|^2+|s_+^{(2)}(0,t)|^2
\leq{}&
\widehat C_0\Big(
|s_+^{(1)}(L,t)|^2
+|s_-^{(2)}(0,t)|^2
+|\partial_t^2u_0(t)|^2
\nonumber\\
&\quad
+|p_+^{(1)}(L,t)|^2
+|p_-^{(2)}(0,t)|^2
+|\partial_tu_0(t)|^2
\Big).
\label{I_s-eq2}
\end{align}
We add all terms in $\sum_{j=1}^2(I_0^{(j)}+I_L^{(j)})$ and reorganize them as
$$
\sum_{j=1}^2(I_0^{(j)}+I_L^{(j)}) = I_{r}+I_p+I_s,
$$
where $I_r$ contains the terms $r_{\pm}^{(i)}$:
$$
I_r=\sum_{i=1}^2\Big[A^{(i)} |r_+^{(i)}(0,t)|^2-B^{(i)} |r_-^{(i)}(0,t)|^2
-A^{(i)}h_+^{(i)}(L)|r_+^{(i)}(L,t)|^2+B^{(i)}h_-^{(i)}(L)|r_-^{(i)}(L,t)|^2\Big],
$$
while $I_p$ and $I_s$ contain the terms $p_{\pm}^{(i)}$ and $s_{\pm}^{(i)}$ respectively:
\begin{align*}
I_p = &~\sum_{i=1}^2C^{(i)} \widehat{\lambda}^{(i)}_+(0) |p_+^{(i)}(0,t)|^2 -D^{(i)}\widehat{\lambda}^{(i)}_-(0) |p_-^{(i)}(0,t)|^2 -C^{(i)}h_+^{(i)}(L)\widehat{\lambda}^{(i)}_+(L) |p_+^{(i)}(L,t)|^2 \\[2mm]
&~+ D^{(i)}h_-^{(i)}(L)\widehat{\lambda}^{(i)}_-(L)|p_-^{(i)}(L,t)|^2\\[2mm]
I_s = &~\sum_{i=1}^2C^{(i)} \widehat{\lambda}^{(i)}_+(0) |s_+^{(i)}(0,t)|^2 -D^{(i)}\widehat{\lambda}^{(i)}_-(0) |s_-^{(i)}(0,t)|^2 -C^{(i)}h_+^{(i)}(L)\widehat{\lambda}^{(i)}_+(L) |s_+^{(i)}(L,t)|^2 \\[2mm]
&~+ D^{(i)}h_-^{(i)}(L)\widehat{\lambda}^{(i)}_-(L)|s_-^{(i)}(L,t)|^2. 
\end{align*}
Here the coefficients are defined by
\begin{align*}
&\widehat{\lambda}^{(i)}_+(0)  =\bar{\lambda}^{(i)}_+(0)-\frac{r^{(i)}_+(0,t)+r^{(i)}_-(0,t)}{2},\qquad 
\widehat{\lambda}^{(i)}_-(0)=|\bar{\lambda}^{(i)}_-(0)|+\frac{r^{(i)}_+(0,t)+r^{(i)}_-(0,t)}{2}\\[2mm]
&\widehat{\lambda}^{(i)}_+(L)  =\bar{\lambda}^{(i)}_+(L)+\frac{r^{(i)}_+(L,t)+r^{(i)}_-(L,t)}{2},\qquad 
\widehat{\lambda}^{(i)}_-(L)=|\bar{\lambda}^{(i)}_-(L)|-\frac{r^{(i)}_+(L,t)+r^{(i)}_-(L,t)}{2}.
\end{align*}
For sufficiently small $\epsilon(T)$, we know that these four numbers are all positive.

\textbf{Estimate for $I_r$:} For $r_\pm^{(i)}$ terms, we use \eqref{eq1.13} and compute  
\begin{align}
I_r
\leq&~ A^{(1)} |u_1(t)-\bar{u}_1|^2 
-A^{(1)}h_+^{(1)}(L)|r_+^{(1)}(L,t)|^2+B^{(1)}h_-^{(1)}(L)|r_-^{(1)}(L,t)|^2 \nonumber\\[2mm]
&+A^{(2)} |r_+^{(2)}(0,t)|^2-B^{(2)} |r_-^{(2)}(0,t)|^2
 +B^{(2)}h_-^{(2)}(L)|u_2(t)-\bar{u}_2|^2\nonumber \\[2mm]
 \leq& -C_1|\alpha|^2+C_2|\beta|^2 +C_3\big(|u_1(t)-\bar{u}_1|^2+|u_2(t)-\bar{u}_2|^2\big)\label{I_r-eq1}
\end{align}
Here we define $$\alpha=(r_+^{(1)}(L,t), r_-^{(2)}(0,t)),\qquad \beta=(r_-^{(1)}(L,t), r_+^{(2)}(0,t)).$$
The constants are given by  
\begin{equation}\label{I_r-eq2}
C_1=\min\{A^{(1)}h_+^{(1)}(L), B^{(2)} \} ,\quad C_2=\max\{B^{(1)}h_-^{(1)}(L), A^{(2)} \}
\end{equation}
and $C_3=\max\{A^{(1)}, B^{(2)}h_-^{(2)}(L)\}$.
From \eqref{eq1.11} and \eqref{eq1.12}, we can use the local Lipschitz continuity to conclude that
\begin{equation}\label{beta-au}
|\beta|^2\leq C_0(|\alpha|^2+|u_0-\bar{u}_0|^2)
\end{equation}
for $(\alpha,\beta,u_0)$ sufficiently close to $(0,0,\bar{u}_0)$ with $C_0>0$ a constant.
Substituting this into \eqref{I_r-eq1}, we have 
\begin{align*}
    I_r &\leq (C_2C_0-C_1)|\alpha|^2+C_2C_0 |u_0-\bar{u}_0|^2+C_3\big(|u_1(t)-\bar{u}_1|^2+|u_2(t)-\bar{u}_2|^2\big).
\end{align*}
According to the expression in \eqref{I_r-eq2}, we take $A^{(1)}$, $B^{(2)}$ sufficiently large and $A^{(2)}$, $B^{(1)}$ sufficiently small such that $C_2C_0-C_1$ is negative. Consequently, we obtain the estimate for $r_\pm^{(i)}$ terms.

\textbf{Estimate for $I_p$ and $I_s$:} For $p_\pm^{(i)}$ terms, we compute
\begin{align}
I_p \leq &~ C^{(1)} \widehat{\lambda}^{(1)}_+(0) |\partial_t u_1(t)|^2  -C^{(1)}h_+^{(1)}(L)\widehat{\lambda}^{(1)}_+(L) |p_+^{(1)}(L,t)|^2 + D^{(1)}h_-^{(1)}(L)\widehat{\lambda}^{(1)}_-(L)|p_-^{(1)}(L,t)|^2 \nonumber \\[3mm]
&+ C^{(2)} \widehat{\lambda}^{(2)}_+(0) |p_+^{(2)}(0,t)|^2 -D^{(2)}\widehat{\lambda}^{(2)}_-(0) |p_-^{(2)}(0,t)|^2 + D^{(2)}h_-^{(2)}(L)\widehat{\lambda}^{(2)}_-(L)|\partial_tu_2(t)|^2 \nonumber \\[2mm]
\leq & -\bar{C}_1\Big(|p_+^{(1)}(L,t)|^2 + |p_-^{(2)}(0,t)|^2 \Big)+\bar{C}_2\Big(|p_-^{(1)}(L,t)|^2+
|p_+^{(2)}(0,t)|^2\Big)\nonumber  \\[2mm]
&+\bar{C}_3\Big(|\partial_tu_1(t)|^2+|\partial_tu_2(t)|^2\Big)\label{I_p-eq1}
\end{align}
with
\begin{align*}
\bar{C}_1&=\min\{C^{(1)}h_+^{(1)}(L)\widehat{\lambda}^{(1)}_+(L) ,~D^{(2)}\widehat{\lambda}^{(2)}_-(0)\}>0,\\[2mm] 
\bar{C}_2&=\max\{D^{(1)}h_-^{(1)}(L)\widehat{\lambda}^{(1)}_-(L) ,~C^{(2)}\widehat{\lambda}^{(2)}_+(0)\}>0,\\[2mm]
\bar{C}_3&=\max\{C^{(1)} \widehat{\lambda}^{(1)}_+(0) ,~D^{(2)}h_-^{(2)}(L)\widehat{\lambda}^{(2)}_-(L)\}>0.
\end{align*}
In \eqref{eq-p1}--\eqref{eq-p2}, $\Phi_1$ and $\Phi_2$ are $C^2$-smooth functions with respect to the components and the solution remains in a compact set. Thus we know that there exists a constant $\bar{C}_0>0$ such that
\begin{align}\label{I_p-eq2}
    |p_-^{(1)}(L,t)|^2+|p_+^{(2)}(0,t)|^2&\leq \bar{C}_0 \Big(|p_+^{(1)}(L,t)|^2 + |p_-^{(2)}(0,t)|^2  + |\partial_tu_0(t)|^2\Big).
\end{align}

Similarly, for $s_\pm^{(i)}$ terms, we have 
\begin{align}
I_s
\leq &~ C^{(1)} \widehat{\lambda}^{(1)}_+(0) |\partial_t^2 u_1(t)|^2  -C^{(1)}h_+^{(1)}(L)\widehat{\lambda}^{(1)}_+(L) |s_+^{(1)}(L,t)|^2 + D^{(1)}h_-^{(1)}(L)\widehat{\lambda}^{(1)}_-(L)|s_-^{(1)}(L,t)|^2\nonumber\\[3mm]
&+ C^{(2)} \widehat{\lambda}^{(2)}_+(0) |s_+^{(2)}(0,t)|^2 -D^{(2)}\widehat{\lambda}^{(2)}_-(0) |s_-^{(2)}(0,t)|^2 + D^{(2)}h_-^{(2)}(L)\widehat{\lambda}^{(2)}_-(L)|\partial_t^2u_2(t)|^2\nonumber\\[2mm]
\leq & -\bar{C}_1\Big(|s_+^{(1)}(L,t)|^2 + |s_-^{(2)}(0,t)|^2 \Big)+\bar{C}_2\Big(|s_-^{(1)}(L,t)|^2+
|s_+^{(2)}(0,t)|^2\Big)\nonumber\\[2mm]
&+\bar{C}_3\Big(|\partial_t^2u_1(t)|^2+|\partial_t^2u_2(t)|^2\Big).\label{I_s-eq1}
\end{align}
Combining \eqref{I_s-eq2},\eqref{I_p-eq1}--\eqref{I_s-eq1},  we have 
\begin{align*}
I_p +I_s
\leq & ~(\widehat{C}_0\bar{C}_2-\bar{C}_1)\Big(|s_+^{(1)}(L,t)|^2 + |s_-^{(2)}(0,t)|^2 \Big) \\[2mm]
&+(\widehat{C}_0\bar{C}_2+\bar{C}_0\bar{C}_2-\bar{C}_1)\Big(|p_+^{(1)}(L,t)|^2 + |p_-^{(2)}(0,t)|^2 \Big)\\[2mm]
&+\bar{C}_3\Big(|\partial_tu_1(t)|^2+|\partial_tu_2(t)|^2+|\partial_t^2u_1(t)|^2+|\partial_t^2u_2(t)|^2\Big)\\[2mm]
&+\bar{C}_2(\bar{C}_0+\widehat{C}_0)\Big(|\partial_tu_0(t)|^2+|\partial_t^2u_0(t)|^2\Big).
\end{align*}
According to the expression of $\bar{C}_1$ and $\bar{C}_2$, we can choose $C^{(1)}$, $D^{(2)}$ to be sufficiently large and $C^{(2)}$, $D^{(1)}$ to be sufficiently small such that 
$$
\widehat{C}_0\bar{C}_2-\bar{C}_1 <
\widehat{C}_0\bar{C}_2+\bar{C}_0\bar{C}_2-\bar{C}_1 <0.
$$
This gives the estimate for $I_p$ and $I_s$. We complete the proof.
\end{proof}

Note that our Lyapunov functional $\mathcal{E}(t)$ contains the time derivatives for $r^{(i)}(x,t)$.
In order to relate them to the space
derivatives, we give without proof the following technical lemma \cite{MR3049647}
\begin{lemma}
Let a finite time $T > 0$ be given and the space-time domain $[0, L]\times [0, T ]$.
There exist a positive constant $C>0$ and a real number $\epsilon > 0$, such that for
$r_{\pm} = (r_+, r_-)^T \in (C^1(\Omega))^2$ satisfying \eqref{gas-r-eq1}--\eqref{gas-r-eq2} and $\|r_\pm\|_{C^1([0,T]\times [0,L])}<\epsilon$, the following inequalities hold:
\begin{align*}
\|\partial_xr_\pm\|_{L^2(0,L)}\leq &~C(\|\partial_tr_\pm\|_{L^2(0,L)}+\|r_\pm\|_{L^2(0,L)})\\[2mm]
\|\partial_tr_\pm\|_{L^2(0,L)}\leq &~C(\|\partial_xr_\pm\|_{L^2(0,L)}+\|r_\pm\|_{L^2(0,L)})\\[2mm]
\|\partial_{tt}r_\pm\|_{L^2(0,L)}\leq &~C(\|\partial_{xx}r_\pm\|_{L^2(0,L)}+\|\partial_{x}r_\pm\|_{L^2(0,L)}+\|r_\pm\|_{L^2(0,L)})\\[2mm]
\|\partial_{xx}r_\pm\|_{L^2(0,L)}\leq &~C(\|\partial_{tt}r_\pm\|_{L^2(0,L)}+\|\partial_{t}r_\pm\|_{L^2(0,L)}+\|r_\pm\|_{L^2(0,L)}).
\end{align*}
\end{lemma}
Thanks to this lemma, we know that there exists constants $c_1,c_2>0$ such that 
\begin{equation}\label{eq-equiv}
c_1 \sum_{j=1}^2\|r^{(i)}(\cdot,t)\|^2_{H^2(0,L)} \leq \mathcal{E}(t) \leq c_2 \sum_{j=1}^2\|r^{(i)}(\cdot,t)\|^2_{H^2(0,L)}.
\end{equation}

\subsection{Constructed  control and cheap control property} 
\label{subsec:connecting-control}

The constant stationary control $u_i(t)\equiv\bar u_i$ is not
necessarily compatible with the initial data. We therefore construct
a feasible control that connects the boundary values
compatible with the initial state to the stationary controls over a
fixed time interval.

By the assumption for initial values, we have
\[
u_i(0) = \bar{u}_i+u_i^{\mathrm{in}},\qquad \partial_t u_i(0)=\dot u_i^{\mathrm{in}},
\qquad i=0,1,2,
\]  
such that the zeroth- and first-order compatibility conditions hold.
Since the coupling functions $\Phi_1$ and $\Phi_2$ are smooth in a
neighborhood of the stationary state, by the trace estimate, there exists a
constant $C_{\mathrm{comp}}>0$ such that
\begin{align}
\sum_{i=0}^2
\left(
|u_i^{\mathrm{in}}|^2
+
|\dot u_i^{\mathrm{in}}|^2
\right)
\leq
C_{\mathrm{comp}}
\sum_{j=1}^2
\|r_0^{(j)}\|_{H^2(0,L)}^2.
\label{eq:compatible-data-estimate}
\end{align} 
Let $\tau_c>0$ be fixed independently of the time horizon $T$.
We now introduce two functions $\eta_0,\eta_1\in H^2(0,\infty)$
supported in $[0,\tau_c]$. For $0\leq t\leq\tau_c$, let
\begin{align}
\eta_0(t)
&=
1-3\left(\frac{t}{\tau_c}\right)^2
+2\left(\frac{t}{\tau_c}\right)^3,
\label{eq:eta0}\\
\eta_1(t)
&=
t\left(1-\frac{t}{\tau_c}\right)^2,
\label{eq:eta1}
\end{align}
and set $\eta_0(t)=\eta_1(t)=0$ for $t\geq\tau_c$. These functions
satisfy
\begin{align}
&\eta_0(0)=1,
&&\eta_0'(0)=0,
&&\eta_0(\tau_c)=\eta_0'(\tau_c)=0,
\nonumber\\[2mm]
&\eta_1(0)=0,
&&\eta_1'(0)=1,
&&\eta_1(\tau_c)=\eta_1'(\tau_c)=0.
\label{eq:eta-properties}
\end{align}
Define the constructed control by
\begin{align}
\widetilde u_i(t)
=
\bar u_i
+
 u_i^{\mathrm{in}} \eta_0(t)
+
\dot u_i^{\mathrm{in}}\eta_1(t),
\qquad i=0,1,2.
\label{eq:connecting-control}
\end{align}
It follows from \eqref{eq:eta-properties} that
\begin{align}
\widetilde u_i(0)
&=\bar{u}_i+u_i^{\mathrm{in}},
\qquad
\partial_t\widetilde u_i(0)
 =\dot u_i^{\mathrm{in}},
\label{eq:connecting-initial}
\end{align}
and
\begin{align}
\widetilde u_i(t)=\bar u_i,
\qquad
t\geq\tau_c.
\label{eq:connecting-stationary}
\end{align}
Thus, $\widetilde u$ satisfies the compatibility conditions at
$t=0$. For sufficiently small initial data, it also satisfies the
control constraint and hence belongs to $\mathcal U_T$.

Since $\eta_0$ and $\eta_1$ are fixed functions depending only on
$\tau_c$, we have
\begin{align}
\|\widetilde u_i-\bar u_i\|_{H^2(0,T)}^2
\leq
C_{\tau_c}
\left(
|u_i^{\mathrm{in}} |^2
+
|\dot u_i^{\mathrm{in}}|^2
\right).
\label{eq:connecting-H2-estimate}
\end{align}
Combining \eqref{eq:compatible-data-estimate} and
\eqref{eq:connecting-H2-estimate}, we obtain
\begin{align}
\sum_{i=0}^2
\|\widetilde u_i-\bar u_i\|_{H^2(0,T)}^2
\leq
C
\sum_{j=1}^2
\|r_0^{(j)}\|_{H^2(0,L)}^2.
\label{eq:connecting-control-cost}
\end{align}

Let $\widetilde r=(\widetilde r^{(1)},\widetilde r^{(2)})$ denote
the state corresponding to $\widetilde u$. By
\eqref{Ly-estimate} and \eqref{lemma3.2-eq1}, its Lyapunov
functional satisfies
\begin{align}
\frac{d}{dt}\mathcal E(t)
\leq
-\mu\mathcal E(t)
+
C\sum_{i=0}^2
\left(
|\widetilde u_i(t)-\bar u_i|^2
+
|\partial_t\widetilde u_i(t)|^2
+
|\partial_t^2\widetilde u_i(t)|^2
\right).
\label{eq:connecting-energy}
\end{align}
Integrating \eqref{eq:connecting-energy} over
$[0,t]$, for $0\leq t\leq\min\{T,\tau_c\}$, yields
\begin{align}
\mathcal E(t)
\leq
\mathcal E(0)
+
C\sum_{i=0}^2
\|\widetilde u_i-\bar u_i\|_{H^2(0,\tau_c)}^2.
\end{align}
Therefore, by \eqref{eq:connecting-control-cost} and the equivalence
\eqref{eq-equiv},
\begin{align}
\sup_{0\leq t\leq\min\{T,\tau_c\}}
\sum_{j=1}^2
\|\widetilde r^{(j)}(\cdot,t)\|_{H^2(0,L)}^2
\leq
C
\sum_{j=1}^2
\|r_0^{(j)}\|_{H^2(0,L)}^2.
\label{eq:connecting-state-short}
\end{align}
In particular,
\begin{align}
\int_0^{\min\{T,\tau_c\}}
\sum_{j=1}^2
\|\widetilde r^{(j)}(\cdot,t)\|_{H^2(0,L)}^2\,dt
\leq
C
\sum_{j=1}^2
\|r_0^{(j)}\|_{H^2(0,L)}^2.
\label{eq:connecting-state-integral-short}
\end{align}
If $T>\tau_c$, then \eqref{eq:connecting-stationary} and
\eqref{eq:connecting-energy} imply
\begin{align}
\frac{d}{dt}\mathcal E(t)
\leq
-\mu\mathcal E(t),
\qquad
t\in[\tau_c,T].
\end{align}
Hence
\begin{align}
\mathcal E(t)
\leq
e^{-\mu(t-\tau_c)}\mathcal E(\tau_c),
\qquad
t\in[\tau_c,T].
\label{eq:connecting-decay}
\end{align}
Using again \eqref{eq-equiv} and
\eqref{eq:connecting-state-short}, we obtain
\begin{align}
\int_{\tau_c}^T
\sum_{j=1}^2
\|\widetilde r^{(j)}(\cdot,t)\|_{H^2(0,L)}^2\,dt
&\leq
C\int_{\tau_c}^T
e^{-\mu(t-\tau_c)}\mathcal E(\tau_c)\,dt
\nonumber\\
&\leq
\frac{C}{\mu}\mathcal E(\tau_c)
\nonumber\\
&\leq
C
\sum_{j=1}^2
\|r_0^{(j)}\|_{H^2(0,L)}^2.
\label{eq:connecting-state-integral-long}
\end{align}
Combining \eqref{eq:connecting-control-cost},
\eqref{eq:connecting-state-integral-short}, and
\eqref{eq:connecting-state-integral-long}, we conclude that
\begin{align}
J_T(\widetilde u)
\leq
C_0
\sum_{j=1}^2
\|r_0^{(j)}\|_{H^2(0,L)}^2,
\label{eq:comparison-cost}
\end{align}
where $C_0>0$ is independent of $T$. If $u^*$ is an optimal
control for $\mathcal P(T)$, then
\[
J_T(u^*)
\leq
J_T(\widetilde u).
\]
Consequently,
\begin{align}
J_T(u^*)
\leq
C_0
\sum_{j=1}^2
\|r_0^{(j)}\|_{H^2(0,L)}^2.
\end{align}
This proves the cheap control property.


\begin{remark}
The cheap control property \eqref{cheap-control-prop} can also be interpreted as the integral turnpike property, see e.g. \cite{gugat2019turnpike}. In our proof, we use the constructed control to achieve this. The cheap control property can also be shown using the exact controllability of the system, which holds for a sufficiently large time horizon, see e.g. \cite{MR2655971}. 
\end{remark}

\begin{remark}\label{remark32}
For any $0<t_0<T$, the optimal control $u=(u_0,u_1,u_2)$ and the corresponding solution $r^{(i)}(x,t)$ also satisfy the inequality 
    \begin{align}\label{shifted-cheap-control}
        &\sum_{i=1}^2\int_{t_0}^T \|r^{(i)}(\cdot,t) \|^2_{H^2(0,L)} dt + \sum_{i=0}^2 \|u_i-\bar{u}_i\|_{H^2(t_0,T)}^2
        \leq C_0 \sum_{i=1}^2 \|r^{(i)}(\cdot,t_0)\|^2_{H^2(0,L)}. 
    \end{align}

It suffices to verify that the constructed control argument can be
restarted at \(t_0\). For the optimal control \(u\) and its
corresponding state \(r^{(i)}\), the cheap-control estimate
\eqref{cheap-control-prop}, the Lyapunov estimate
\eqref{Ly-estimate}--\eqref{lemma3.2-eq1}, and the norm equivalence
\eqref{eq-equiv} imply that
\[
\sum_{i=1}^{2}
\|r^{(i)}(\cdot,t)\|_{H^2(0,L)}^2
\leq
C_*
\sum_{i=1}^{2}
\|r_0^{(i)}\|_{H^2(0,L)}^2,
\qquad t\in[0,T],
\]
where \(C_*>0\) is independent of \(t\) and \(T\).
The initial constant $\epsilon_{in}(T)$ in \eqref{initial data} is chosen sufficiently small 
so that the
well-posedness, Lyapunov estimate, and  constructed control arguments apply for \(r^{(i)}(\cdot,t_0)\) with any $t_0\in[0,T]$.
Moreover, the regularity of the optimal solutions yields the compatibility conditions for \(r^{(i)}(\cdot,t_0)\). 
We now repeat the control construction of
Section~\ref{subsec:connecting-control} on \([t_0,T]\), with initial state
\(r^{(i)}(\cdot,t_0)\), and denote the resulting control by
\(\widetilde u^{\,t_0}\). Define
\[
\widehat u(t)=
\begin{cases}
u(t),&0\leq t\leq t_0,\\
\widetilde u^{\,t_0}(t),&t_0<t\leq T.
\end{cases}
\]
The compatibility conditions at \(t_0\) imply that
\(\widehat u\in H^2(0,T)^3\). Furthermore,
\[
\|\widehat u_i-\bar u_i\|_{H^2(0,T)}^2
\leq
C_{\mathrm{adm}}
\sum_{j=1}^2
\|r_0^{(j)}\|_{H^2(0,L)}^2.
\]
By choosing the initial-data constant $\epsilon_{in}(T)$ sufficiently small, the right-hand side is bounded by
\(\epsilon(T)^2\). Hence, \(\widehat u\in \mathcal U_T\). Since \(u\) is optimal and \(u|_{[0,t_0]}=\widehat u|_{[0,t_0]}\), comparison of their costs gives
\[
J_{[t_0,T]}(u)\leq
J_{[t_0,T]}(\widetilde u^{\,t_0}).
\]
Applying the estimates of Section~3.2 on \([t_0,T]\) yields
\eqref{shifted-cheap-control}.

\end{remark}

\section{Exponential Turnpike property}\label{section:4}

In this section, we state the main theorem:
\begin{theorem}\label{thm14}
Assume that the initial data for the equations \eqref{gas-r-eq1}--\eqref{gas-r-eq2} satisfy the constraint \eqref{initial data} and the compatibility conditions. Then the exponential turnpike property holds for the optimal solution of the optimal control problem $\mathcal{P}(T)$. Namely, there exist constants $K>0$ and $\gamma>0$, which are independent of $T$, such that the optimal solution for $\mathcal{P}(T)$ satisfies:
\begin{align}\label{expo-turnpike}
\sum_{i=1}^2\|r^{(i)}(\cdot,t)\|^2_{H^2(0,L)}  \leq K e^{-\gamma t} \sum_{i=1}^2\|r^{(i)}(\cdot,0)\|^2_{H^2(0,L)}
\end{align}
for any $t\in(0,T)$. 
\end{theorem}
In order to prove the theorem, we first introduce the following lemma
that contains a bound for the  temporal growth of the $H^2$-norm of the   optimal state
in terms of the $H^2$-norm of the initial state. 
\begin{lemma}\label{lemma2.4}
Let $t_1$ and $t_2$ be given time points satisfying $0\leq t_1 < t_2\leq T$. 
There exists a constant $C_M>0$ 
(independent of $T$) 
such that the  following bound holds for the optimal state
\begin{align*}
\sum_{i=1}^2 \|r^{(i)}(\cdot,t_2) \|_{H^2(0,L)}^2  \leq C_M \sum_{i=1}^2  \|r^{(i)}(\cdot,t_1) \|_{H^2(0,L)}^2.
\end{align*}    
\end{lemma}

\begin{proof}
By \eqref{Ly-estimate} and \eqref{lemma3.2-eq1}, we know that 
\begin{align*}
\frac{d}{dt}\mathcal{E}(t) \leq&~ -\mu \,  \mathcal{E}(t)+\sum_{j=1}^2\left(I^{(j)}_0+I^{(j)}_L\right)\\
\leq &~C\sum_{i=0}^2\Big(|u_i(t)-\bar{u}_i|^2+|\partial_tu_i(t)|^2+|\partial_t^2u_i(t)|^2\Big).
\end{align*}
We integrate over $t\in[t_1,t_2]$ and obtain
$$
\mathcal{E}(t_2)\leq \mathcal{E}(t_1) + C\sum_{i=0}^2\|u_i(t)-\bar{u}_i\|_{H^2(t_1,t_2)}^2\leq \mathcal{E}(t_1) + C\sum_{i=0}^2\|u_i(t)-\bar{u}_i\|_{H^2(t_1,T)}^2.
$$
By the cheap control property (see \eqref{cheap-control-prop} and Remark \ref{remark32}), we have
$$
\mathcal{E}(t_2) \leq \mathcal{E}(t_1) + CC_0 \sum_{i=1}^2 \|r^{(i)}(\cdot,t_1)\|^2_{H^2(0,L)}.
$$
Furthermore, we use \eqref{eq-equiv} to get
$$
c_1 \sum_{j=1}^2\|r^{(i)}(\cdot,t_2)\|^2_{H^2(0,L)} \leq \mathcal{E}(t_2) \leq c_2 \sum_{j=1}^2\|r^{(i)}(\cdot,t_1)\|^2_{H^2(0,L)}  + CC_0 \sum_{i=1}^2 \|r^{(i)}(\cdot,t_1)\|^2_{H^2(0,L)}.
$$
At last, we take the constant $C_M=(c_2+CC_0)/c_1$ and complete the proof.
\end{proof}
\noindent 
From the above lemma about the bounded growth of the optimal state, we proceed with
a  bound for the norms of the optimal state in the time-intervals $[n\, \tau,  \, T]$,
where
$\tau \in (0, \, T)$ and  $n$ is an integer such that $n\tau <T$.

\begin{lemma}\label{lemma13}
The following inequality holds for any $t \in [n\tau,T]$ with a given constant $\tau>0$ and an integer $1\leq n\leq \lfloor T/\tau\rfloor $:
\begin{equation}
\label{yz}
\sum_{i=1}^2 \|r^{(i)}(\cdot,t) \|_{H^2(0,L)}^2 \leq \left(\frac{C_0C_M}{\tau}\right)^n \sum_{i=1}^2 \|r^{(i)}(\cdot,0) \|_{H^2(0,L)}^2.
\end{equation}
\end{lemma}
\begin{proof}
    We first prove the case  $n=1$. There exists a point $t_1\in[0,\tau]$ such that
    \begin{align*}
    \sum_{i=1}^2 \|r^{(i)}(\cdot,t_1) \|_{H^2(0,L)}^2 \leq \frac{1}{\tau} \int_{0}^{\tau}
    \sum_{i=1}^2 \|r^{(i)}(\cdot,t) \|_{H^2(0,L)}^2 dt
    \leq \frac{C_0}{\tau} \sum_{i=1}^2 \|r^{(i)}(\cdot,0) \|_{H^2(0,L)}^2.
    \end{align*}
    Note that the second inequality is due to the cheap control property.
For any $t\geq \tau \geq t_1$, we obtain by Lemma \ref{lemma2.4} about the bounded growth of the optimal state 
    \begin{align*}
    \sum_{i=1}^2 \|r^{(i)}(\cdot,t) \|_{H^2(0,L)}^2
    \leq C_M\sum_{i=1}^2 \|r^{(i)}(\cdot,t_1) \|_{H^2(0,L)}^2
    \leq \frac{C_0C_M}{\tau} \sum_{i=1}^2 \|r^{(i)}(\cdot,0) \|_{H^2(0,L)}^2.
    \end{align*}
    Now we proceed by induction.
    Suppose that   inequality  (\ref{yz}) holds for a natural number $n\geq 1$.
    Now we  show  the result for $n+1$. There exists $t_n\in[n\tau,(n+1)\tau]$ such that
\begin{align*}
    \sum_{i=1}^2 \|r^{(i)}(\cdot,t_n) \|_{H^2(0,L)}^2 \leq&~ \frac{1}{\tau} \int_{n\tau}^{(n+1)\tau}\sum_{i=1}^2 \|r^{(i)}(\cdot,t) \|_{H^2(0,L)}^2dt\\[2mm] 
    \leq&~\frac{C_0}{\tau} \sum_{i=1}^2 \|r^{(i)}(\cdot,n\tau) \|_{H^2(0,L)}^2 \leq \frac{C_0}{\tau} \left(\frac{C_0C_M}{\tau}\right)^n \sum_{i=1}^2 \|r^{(i)}(\cdot,0) \|_{H^2(0,L)}^2.
    \end{align*}
Thus for any $t\in[(n+1)\tau,T]$, we obtain by Lemma \ref{lemma2.4}
$$
\sum_{i=1}^2 \|r^{(i)}(\cdot,t) \|_{H^2(0,L)}^2 \leq C_M\sum_{i=1}^2 \|r^{(i)}(\cdot,t_n) \|_{H^2(0,L)}^2\leq \left(\frac{C_0C_M}{\tau}\right)^{n+1} \sum_{i=1}^2 \|r^{(i)}(\cdot,0) \|_{H^2(0,L)}^2
$$ 
and this completes the proof.
\end{proof}

Now we prove the exponential turnpike theorem.  
\begin{proof}[Proof of Theorem \ref{thm14}]
In this proof, we fix the constant $\tau$ in Lemma \ref{lemma13} such that 
\[\tau>C_0 \, C_M.\] 
It suffices to prove \eqref{expo-turnpike} for sufficiently large $t>\tau$, since for $t\leq \tau$,
inequality \eqref{expo-turnpike} holds if we choose the constant $K$ sufficiently large due to Lemma \ref{lemma2.4}.

    So now we consider the case $T>\tau$. 
For any $t \in [\tau, T)$, define $n = \lfloor t/\tau \rfloor$. Then $1 \leq n \leq \frac{T}{\tau}$ and $t \in [n \, \tau, T)$. By Lemma~\ref{lemma13}, we obtain
$$
\sum_{i=1}^2 \|r^{(i)}(\cdot,t)\|_{H^2(0,L)}^2 
\leq \left(\frac{C_0 C_M}{\tau}\right)^n 
\sum_{i=1}^2 \|r^{(i)}(\cdot,0)\|_{H^2(0,L)}^2.
$$

From the definition of $n$, we have $n > t/\tau - 1$. Moreover, since $\tau > C_0 \, C_M$, it follows that
$$
\left(\frac{C_0 C_M}{\tau}\right)^n 
= \left(\frac{\tau}{C_0 C_M}\right)^{-n}
\leq \left(\frac{\tau}{C_0 C_M}\right)^{1 - t/\tau}.
$$
Therefore, we obtain the exponential estimate
$$
\sum_{i=1}^2 \|r^{(i)}(\cdot,t)\|_{H^2(0,L)}^2 
\leq K \,  e^{-\gamma \, t} 
\sum_{i=1}^2 \|r^{(i)}(\cdot,0)\|_{H^2(0,L)}^2,
\quad \forall\, t \in [\tau, T),
$$
where
$$
K = \frac{\tau}{C_0 \,  C_M}, 
\qquad 
\gamma = \frac{1}{\tau} \log\!\left(\frac{\tau}{C_0 \,  C_M}\right) > 0.
$$
\end{proof}

\section{Numerical experiment}\label{section:5}

In this section, we provide a numerical example to visualize the exponential turnpike result \eqref{expo-turnpike}. Consider the network with two pipes and one compressor station shown in \hyperref[fig1]{\textit{Figure \ref*{fig1}}}. The gas network parameters given in \hyperref[tab:networkParameters]{\textit{Table \ref*{tab:networkParameters}}} are close to real world values (see, e.g., \cite{GugatSchuster2018, GugatEtAl2024, SchusterEtAl2026}), for the readers' convenience we apply the same parameters for both pipes.

\begin{table}[htbp]
	\centering
	\begin{tabular}{l l l l}
		\toprule
		Variable & Letter & Value & Unit \\
		\midrule 
		time horizon & $T$ & $12$ & h \\
		pipe friction coefficient & $f_g$ & $0.0137$ &  \\
		pipe diameter & $D$ & $0.5$ & m \\
		pipe length & $L$ & $55$ & km \\
		universal gas constant & $R$ & $8.3145$ & J/(mol K)$^{-1}$ \\
		gas temperature & $Temp$ & $283.15$ & K \\
		isentropic exponent & $\kappa$ & $0.5$ & \\
		intransient inlet density at pipe $1$ & $\bar{\rho}_0$ & $46.78$ & kg/m$^{3}$ \\
		intransient gas outflow at pipe $2$ & $\bar{q}_0$ & $200$ & kg/(m$^2$ s) \\
		\bottomrule
	\end{tabular}
	\caption{Parameters and coefficients for the example with two pipes and one compressor stations shown in \hyperref[fig1]{\textit{Figure \ref*{fig1}}}}
	\label{tab:networkParameters}
\end{table}

Note that the instationary inlet density $\bar{\rho}_0$ coincides with a inlet pressure of $65$ bar for ideal gases. For this example, we consider natural gas as ideal gas, that consists of $90\%$ methane, $6\%$ ethane and $4\%$ propane. The specific gas constant is defined by $R/M$, where $R$ is the universal gas constant and $M$ is the molar mass. This yields a specific gas constant $R_S$ for natural gas of
\begin{equation*}
	R_S = 490.69 \frac{\text{J}}{\text{kg}\cdot\text{K}}.
\end{equation*}
For ideal gas, the sound speed in the gas $a$ is defined by (see, e.g., \cite{Gas-Modell:DomschkeHillerLangTischendorf2017, GugatSchuster2018})
\begin{equation*}
	a^2 = R_S \cdot Temp \quad \Leftrightarrow \quad a = 372.75 \frac{\text{m}}{\text{s}}.
\end{equation*}
As initial condition, we apply the steady state of \eqref{gas-eq1}, \eqref{gas-eq2} with the boundary conditions $\bar{\rho}^{(1)}(0) = 43.18$ kg/m$^3$ (which coincides with a pressure of $60$ bar), $\bar{q}^{(2)}(L) = 180$ kg/m$^2$s and with the initial compressor control $\bar{u}_0 = 1$. For the optimal control problem, we slightly adapt the objective function \eqref{objective-func} by adding weights to the $H^2$-norms. The weights are introduced for numerical scaling purposes, compensating for the different magnitudes of the individual terms in the $H^2$-norm. Specifically, the $L^2$-norm the Riemann invariants are quite large compared to the $L^2$-norm of their derivatives on the time horizon of $T = 12$h$ = 43200$s. Alternatively, such scaling effects can be incorporated by a suitable non-dimensionalization of the state and control variables and of the space-time coordinates (see, e.g., \cite{SakamotoSchuster2025}). The objective function \eqref{objective-func} is given by
\begin{equation*}
	\begin{aligned}
		J_T(u_0, u_1, u_2) &= \sum_{i=1}^2\int_0^T \|r^{(i)}(\cdot,t)\|^2_{H^2(0,L)}\ dt + \sum_{i=0}^2 \|u_i-\bar{u}_i\|_{H^2(0,T)}^2 \\
		&= \sum_{i=1}^2\int_0^T \|r^{(i)}(\cdot,t)\|^2_{L^2(0,L)} + \left\| \frac{\partial}{\partial x}	r^{(i)}(\cdot,t) \right\|^2_{L^2(0,L)} + \left\| \frac{\partial^2}{\partial x^2}	r^{(i)}(\cdot,t) \right\|^2_{L^2(0,L)}\ dt \\
		&\quad+ \|u_0-\bar{u}_0\|_{L^2(0,T)}^2 + 
        \left\| \frac{\partial}{\partial t} u_0 \right\|_{L^2(0,T)}^2 + \left\| \frac{\partial^2}{\partial t^2} u_0 \right\|_{L^2(0,T)}^2 \\
		&\quad+ \sum_{i=1}^2 \|u_i-\bar{u}_i\|_{L^2(0,T)}^2 + \left\| \frac{\partial}{\partial t} u_i \right\|_{L^2(0,T)}^2 + \left\| \frac{\partial^2}{\partial t^2} u_i \right\|_{L^2(0,T)}^2.
	\end{aligned}
\end{equation*}
In order to 
obtain a meaningful result with realistic physical data, we define the adapted objective function
\begin{equation*}
	\begin{aligned}
		\tilde{J}_T(u_0, u_1, u_2) &= \sum_{i=1}^2\int_0^T  \omega_1\|r^{(i)}(\cdot,t)\|^2_{L^2(0,L)} + 
        \omega_2 \left\| \frac{\partial}{\partial x}	r^{(i)}(\cdot,t)\right\|^2_{L^2(0,L)} \\
		&\hspace{5.15cm}+ \omega_3 \left\| \frac{\partial^2}{\partial x^2} r^{(i)}(\cdot,t) \right\|^2_{L^2(0,L)}\ dt \\
		&\quad+ \omega_4 \|u_0-\bar{u}_0\|_{L^2(0,T)}^2 + \omega_5
        \left\| \frac{\partial}{\partial t} u_0 \right\|_{L^2(0,T)}^2 + \omega_6
        \left\| \frac{\partial^2}{\partial t^2} u_0 \right\|_{L^2(0,T)}^2 \\
		&\quad+ \sum_{i=1}^2 \omega_7\|u_i-\bar{u}_i\|_{L^2(0,T)}^2 + \omega_8
        \left\| \frac{\partial}{\partial t} u_i \right\|_{L^2(0,T)}^2 + \omega_9
        \left\| \frac{\partial^2}{\partial t^2} u_i \right\|_{L^2(0,T)}^2,
	\end{aligned}
\end{equation*}
with the weights 
\begin{equation*}
	\begin{aligned}
		&\omega_1 = 10^{-9}, & \qquad &\omega_4 = 1, & \qquad &\omega_7 = 10^{-5}\\
		&\omega_2 = 1, & &\omega_5 = 10^5, & &\omega_8 = 1,\\
		&\omega_3 = 10^9, & &\omega_6 = 10^{10}, & &\omega_9 = 10^5,
	\end{aligned}
\end{equation*}
that take into account the different orders of magnitude of
the state variables (density and flux) and the Riemann invariants.
%
To be precise, 
we apply different weights to $u_0$ than to $u_1$, $u_2$ since $u_0$ is defined for the coupling of the densities, while $u_1, u_2$ are defined as boundary conditions for the Riemann invariants. \\

For the numerical solution of the PDE system \eqref{gas-RI-eq1}, an upwind-downwind scheme with $12$ discretization points for the spatial derivative (i.e., a discretization in intervals of length $5$ km) together with an implicit Euler method with $73$ discretization points for the time derivative (i.e., a discretization in intervals of length $10$ min) was applied. For the numerical optimization we used the \textit{AMPL} software package with the open-source interior point solver \textit{IPOPT} and the linear sparse systems solver \textit{MUMPS} (see, e.g., \cite{ampl2002}). The exponential turnpike property \eqref{expo-turnpike} is shown in \hyperref[fig:expTP]{\textit{Figure \ref*{fig:expTP}}}, the controls are shown in \hyperref[fig:controls]{\textit{Figure \ref*{fig:controls}}} and the corresponding densities and mass fluxes at the boundaries are shown in \hyperref[fig:densitiesAndFlows]{\textit{Figure \ref*{fig:densitiesAndFlows}}}. 

\begin{figure}[htbp]
	\centering
	\includegraphics[width=.8\textwidth]{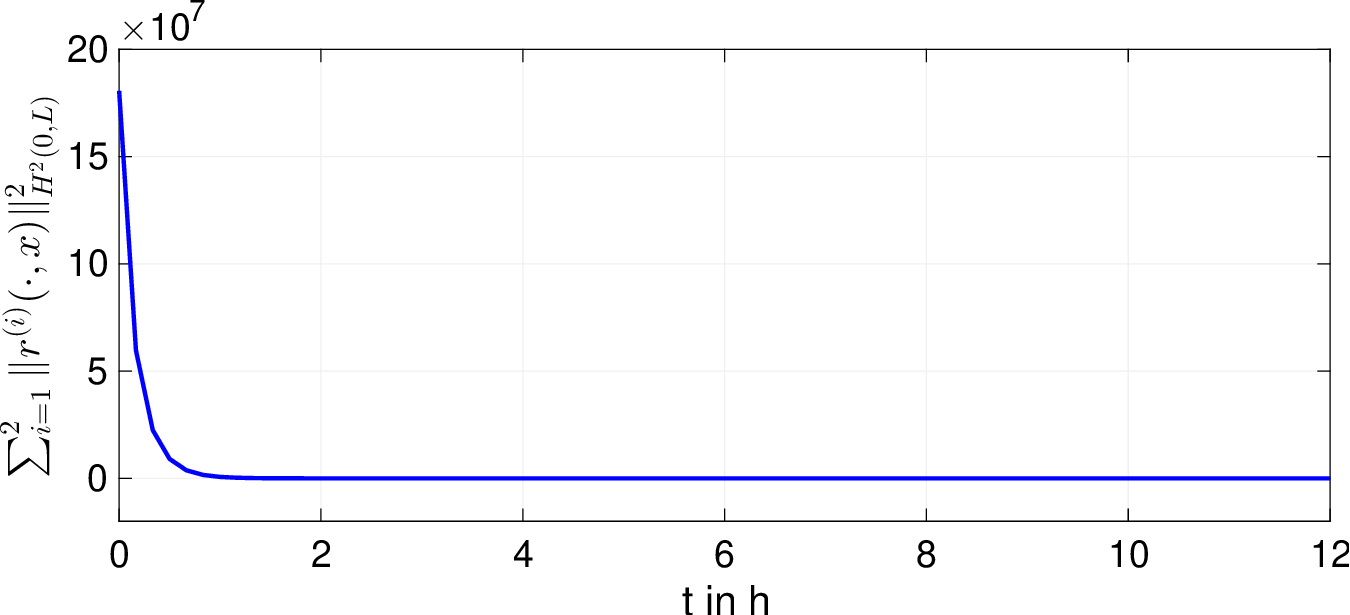}
	\caption{The exponential turnpike property for $r^{(i)}$ as stated in \eqref{expo-turnpike}}
	\label{fig:expTP}
\end{figure}

\begin{figure}[htbp]
	\centering
	\includegraphics[width=.8\textwidth]{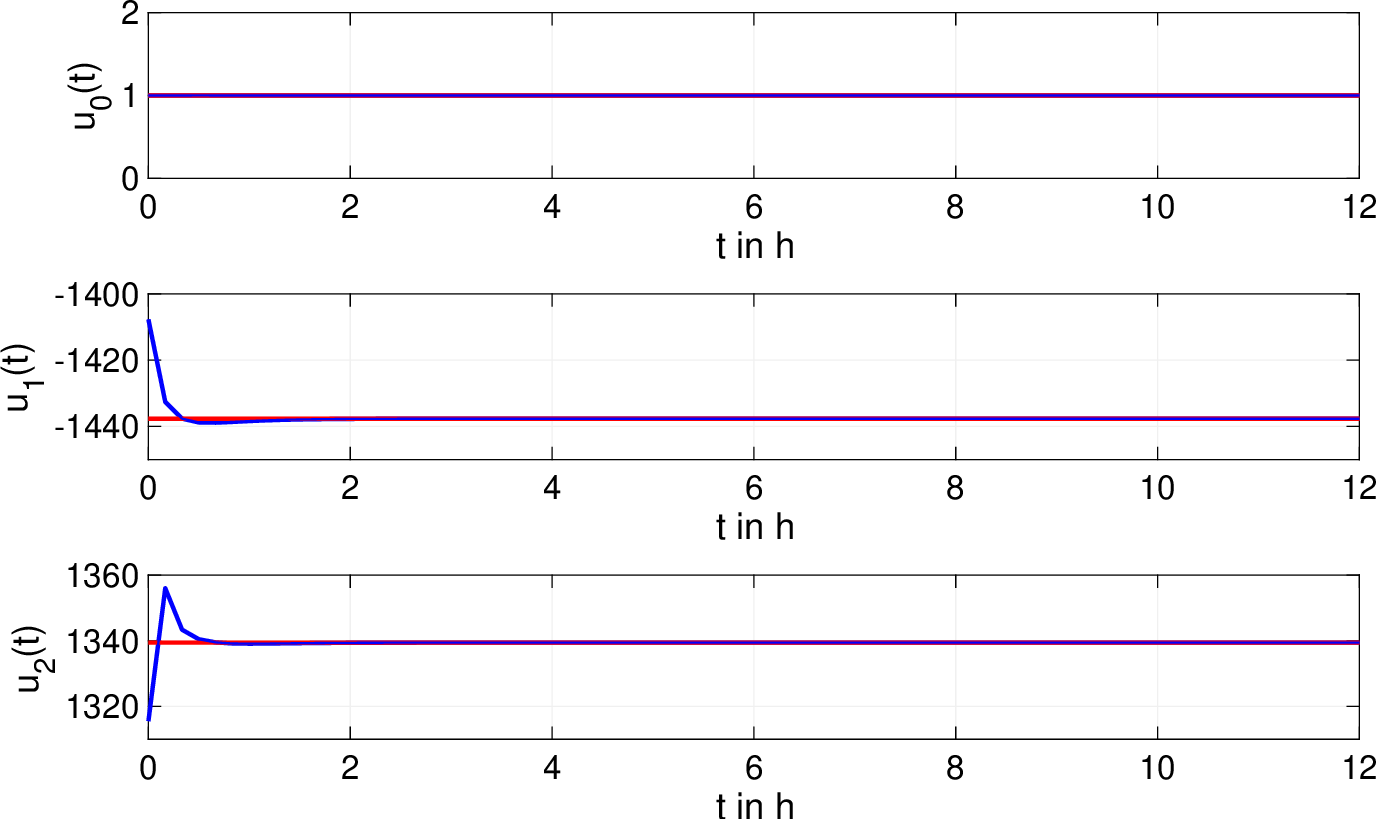}
	\caption{Compressor control $u_0$ (cf. \eqref{couplingCompressor}) and boundary controls $u_1, u_2$ (cf. \eqref{gas-RI-eq4}, \eqref{gas-RI-eq5})}
	\label{fig:controls}
\end{figure}

\begin{figure}[htbp]
	\centering
	\includegraphics[width=\textwidth]{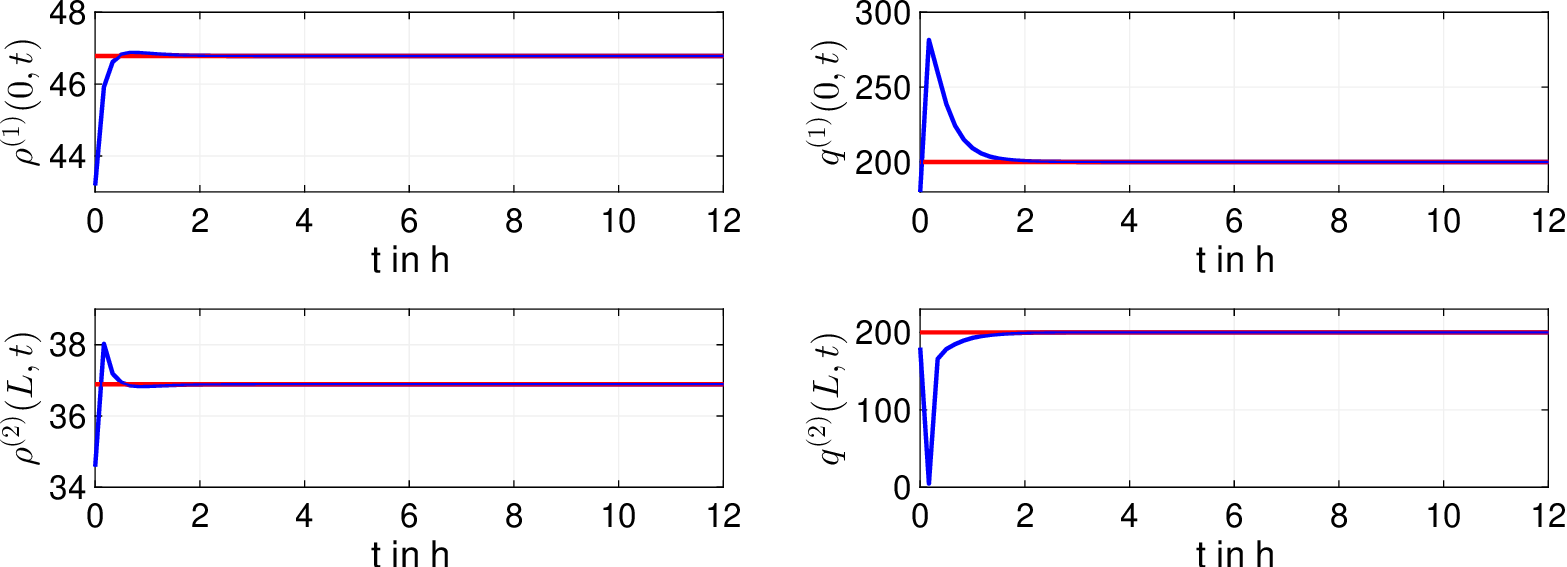}
	\caption{Incoming densities and mass fluxes $\rho^{(1)}(0,t),q^{(1)}(0,t)$ at pipe $1$ and outgoing densities and mass fluxes $\rho^{(2)}(L,t), q^{(2)}(L,t)$ at pipe 2}
	\label{fig:densitiesAndFlows}
\end{figure}

\hyperref[fig:expTP]{\textit{Figure \ref*{fig:expTP}}} illustrates the exponential decay predicted by \hyperref[thm14]{\textit{Theorem \ref*{thm14}}}. After a short initial transient, the optimal trajectory rapidly approaches the prescribed stationary solution and remains close to it over the considered time horizon. This behavior is also reflected in the individual controls in \hyperref[fig:controls]{\textit{Figure \ref*{fig:controls}}}. In particular, not only the deviation
\begin{equation*}
	r^{(i)}= R^{(i)}-\bar{R}^{(i)},
\end{equation*}
for which the exponential turnpike estimate \eqref{expo-turnpike} is established, but also the compressor and boundary controls approach their corresponding stationary values. Moreover, \hyperref[fig:densitiesAndFlows]{\textit{Figure \ref*{fig:densitiesAndFlows}}} shows that the same qualitative behavior is inherited by the physically relevant variables. The densities and mass fluxes at the network boundaries rapidly approach their stationary values. This is consistent with the fact that, locally around the considered subsonic steady state, the physical variables $(\rho, q)$ depend smoothly on the Riemann invariants, while the boundary and compressor controls are linked to their traces through the boundary and coupling conditions. Thus, the numerical experiment indicates that the turnpike behavior established analytically in terms of the Riemann-invariant deviations manifests itself directly in the quantities relevant for the operation of the gas network.

\section{Conclusions}\label{section:6}
We have proved an exponential turnpike property for an optimal boundary control problem 
for a system that is governed by a quasi-linear hyperbolic
partial differential equation 
modeling compressor-driven gas flow in pipelines. The analysis is carried out in the framework of $H^2$-solutions and applies to non-constant steady states arising in a nodal control setting.

The proof combines the existence of optimal controls for small perturbations, a cheap control property derived via a Lyapunov approach, and a constructed control. 
This exponentially stabilizing control also yields an
upper bound for the growth of the norm of the  optimal state.
These ingredients are then integrated into the general turnpike framework of \cite{MR4955417} to obtain the exponential turnpike result.
Numerical simulations illustrate the theoretical findings.

\section*{Acknowledgment}
This work was supported by DFG in the  Project  
\emph{Regelung interagierender Partikelsysteme und ihre
probabilistischen und fluid-dynamischen Beschreibungen},
Projektnummer: 560288187.

\bibliographystyle{plain}
\bibliography{ref}



\end{document}